\documentclass[twoside,12pt,reqno]{amsart}
\usepackage{amscd, bbm, amsaddr, mathtools}
\usepackage{amssymb,mathabx, enumitem}
\usepackage{mathrsfs,wasysym, tikz, tikz-cd, extpfeil}
\usetikzlibrary{matrix, arrows}
\usepackage{amsmath}

\usepackage[margin=2cm]{geometry}

\usepackage[pdftex,colorlinks,hypertexnames=false,linkcolor=black,urlcolor=black,citecolor=black]{hyperref}

\makeatletter
\@addtoreset{equation}{section}

\numberwithin{equation}{section}

\newtheorem{Theorem}{Theorem}[section]
\newtheorem*{Theorem*}{Theorem}
\newtheorem*{Corollary*}{Corollary}

\newtheorem{Lemma}[Theorem]{Lemma}
\newtheorem{Proposition}[Theorem]{Proposition}
\newtheorem{Corollary}[Theorem]{Corollary}

\theoremstyle{definition}

\theoremstyle{remark}
\newtheorem{Remark}[Theorem]{Remark}
\newtheorem*{proofof}{Proof of Theorem~\ref{T:lastbit}}
\newtheorem*{Remark*}{Remark}

\newbox\squ  
\setbox\squ=\hbox{\vrule width.6pt
   \vbox{\hrule height.6pt width.4em\kern1ex
         \hrule height.6pt}%
   \vrule width.6pt}

\newcommand{\C}{\mathbb{C}}
\newcommand{\N}{\mathbb{N}}
\newcommand{\Z}{\mathbb{Z}}

\newcommand{\A}{\mathbb{A}}
\renewcommand{\Z}{\mathbb{Z}}

\newcommand{\g}{\mathfrak{g}}
\newcommand{\gl}{\mathfrak{gl}}
\newcommand{\p}{\mathfrak{p}}

\renewcommand{\sl}{\mathfrak{sl}}
\newcommand{\so}{\mathfrak{so}}
\renewcommand{\sp}{\mathfrak{sp}}

\newcommand{\n}{\mathfrak{n}}

\newcommand{\Nc}{\mathcal{N}}
\newcommand{\tN}{\widetilde{\Nc}}
\renewcommand{\O}{\mathcal{O}}

\newcommand{\Ss}{\mathcal{S}}

\renewcommand{\A}{\mathcal{A}}
\newcommand{\Q}{\mathcal{Q}}

\newcommand{\ad}{\operatorname{ad}}
\newcommand{\Ad}{\operatorname{Ad}}

\newcommand{\Lie}{\operatorname{Lie}}

\newcommand{\Hom}{\operatorname{Hom}}
\newcommand{\GL}{\operatorname{GL}}

\newcommand{\SO}{\operatorname{SO}}
\newcommand{\OO}{\operatorname{O}}
\newcommand{\Sp}{\operatorname{Sp}}
\newcommand{\gr}{\operatorname{gr}}

\newcommand{\Spec}{\operatorname{Spec}}

\newcommand{\Out}{\operatorname{Out}}
\newcommand{\Aut}{\operatorname{Aut}}

\newcommand{\Cas}{\operatorname{Cas}}

\renewcommand{\Out}{\operatorname{Out}}

\newcommand{\codim}{\operatorname{codim}}

\newcommand{\id}{{\operatorname{id}}}

\newcommand{\Qnt}{\operatorname{Q}}
\newcommand{\PD}{\operatorname{PD}}
\newcommand{\GrAlg}{\operatorname{\bf{G}}}
\newcommand{\FAlg}{\operatorname{\bf{F}}}
\newcommand{\SFAlg}{\operatorname{\bf{SF}}}
\newcommand{\Sets}{\operatorname{\bf{Sets}}}
\newcommand{\graded}{\operatorname{\tt{gr}}}
\newcommand{\tw}{\operatorname{tw}}
\newcommand{\reg}{\operatorname{reg}}
\newcommand{\filt}{\operatorname{filt}}

\newcommand{\CA}{\operatorname{CA}}

\newcommand{\op}{\bar{\p}}
\newcommand{\on}{\bar{\n}}

\newcommand{\teta}{\tilde\eta}

\renewcommand{\i}{\iota}

\def\into{\hookrightarrow}
\def\onto{\twoheadrightarrow}

\title[]{\boldmath  
Equivariant deformation theory for nilpotent slices in symplectic Lie algebras}

\author{Filippo Ambrosio and Lewis Topley}

\email{lt803@bath.ac.uk}
\email{filippo.ambrosio@uni-jena.de}

\begin{document}

\begin{abstract}
The Slodowy slice is a flat Poisson deformation of its nilpotent part, and it was demonstrated by Lehn--Namikawa--Sorger that there is an interesting infinite family of nilpotent orbits in symplectic Lie algebras  for which the slice is not the universal Poisson deformation of its nilpotent part.
This family corresponds to slices to nilpotent orbits in symplectic Lie algebras whose Jordan normal form has two blocks.
We show that the nilpotent Slodowy varieties associated to these orbits are isomorphic as Poisson $\C^\times$-varieties to nilpotent Slodowy varieties in type {\sf D}. 
It follows that the universal Poisson deformation in type {\sf C} is a slice in type {\sf D}.

When both Jordan blocks have odd size the underlying singularity is equipped with a $\Z_2$-symmetry coming from the type {\sf D} realisation. We prove that the Slodowy slice in type {\sf C} is the $\Z_2$-equivariant universal Poisson deformation of its nilpotent part. This result also has non-commutative counterpart, identifying the finite $W$-algebra as the universal equivariant quantization.
\end{abstract}

\maketitle

\section{Introduction}
It follows from recent work of Namikawa \cite{Na1} that every conical symplectic singularity $X$ admits a universal Poisson deformation, over an affine base space. This means that every Poisson deformation can be obtained from it via a unique base change.
By the work of Losev \cite{Lo} this Poisson deformation admits a quantization which satisfies another remarkable universal property: every quantization of a Poisson deformation can be obtained from this quantization via base change (see Section~\ref{S:CSS} for more detail). A convenient formalism for expressing these facts is the functor of Poisson deformations and the functor of quantizations of Poisson deformations. 
The work of Namikawa and Losev can be succinctly expressed by saying that these functors are representable, that they are represented over the same base, and satisfy an excellent compatibility condition (see \cite[Proposition~3.5]{Lo} and \cite[Definition~2.12]{ACET}).

Let $G$ be a simple complex algebraic group with Lie algebra $\g$.
If $e$ is an element in the nilpotent cone $\Nc(\g) \subset \g$ with  adjoint orbit $\O$ then we can consider the Slodowy slice $\Ss_e$ which is transversal to $\O$ at $e$, and the nilpotent Slodowy variety $\Nc_e := \Nc(\g) \cap \Ss_e$. 
This subvariety comes equipped with a natural $\C^\times$-action contracting to $e$, and it acquires a conical Poisson structure by Hamiltonian reduction, as observed by Gan--Ginzburg \cite{GG}.
The Springer resolution $\widetilde{\Nc}(\g) \to \Nc(\g)$ restricts to a symplectic resolution $\tN_e \to \Nc_e$.
In summary, $\Nc_e$ is  a conical symplectic singularity.

Similarly $\Ss_e$ is a Poisson variety by Hamiltonian reduction, and the adjoint quotient $\g \to \g/\!/ G$ restricts to $\Ss_e \to \g/\!/ G$ in such a way that the scheme theoretic central fibre of $\Ss_e \to \g/\!/ G$ is equal to $\Nc_e$ (see \cite[\textsection 5]{PrST}).

To summarise the remarks above, $\Ss_e \to \g/\!/G$ is a Poisson deformation of $\Nc_e$.
Furthermore the slice admits a natural quantization over the same base, known as the finite $W$-algebra.
By the uniqueness (up to $G$-conjugacy) of $\sl_2$-triples containing $e$ as the nilpositive, due to Kostant and Mal'cev \cite[\textsection 3.4]{CM}, we see that the Slodowy slice only depends on the adjoint orbit of $e$ up to Poisson isomorphism, and a similar statement holds for finite $W$-algebras (see \cite{GG} for more detail).

Two natural questions arise, for each nilpotent orbit $\O := G\cdot e \subseteq \Nc(\g)$:
\begin{enumerate}
\item {\it is $\Ss_e$ a universal Poisson deformation of $\Nc_e$?}
\item {\it is the finite $W$-algebra a universal filtered quantization of $\Nc_e$?}
\end{enumerate}

The first question was answered comprehensively by Lehn--Namikawa--Sorger \cite{LNS}, who showed that the answer is positive, with a small list of possible exceptions. 
The second question was answered by the current authors and their collaborators: the questions a have positive answer for precisely the same class of orbits \cite[Theorem~1.2]{ACET}.
 The exceptional cases, where $\Ss_e \to \g/\!/ G$ is not a universal Poisson deformation, are listed in the following table.
\begin{table}[ht]
\caption{Cases in which $\Ss_e \to \g/\!/ G$ is not a universal Poisson deformation of $\Nc_e$.}
\begin{tabular}{|c|c|c|c|c|c|c|}
\hline
Type of $\g$ & Any & {\sf BCFG} & {\sf C} & \sf{G} 

\\
\hline
Type of $\O$ & Regular & Subregular & Two Jordan blocks & dimension 8 \\ 
\hline
\end{tabular} \label{table}
\end{table}

In \cite{ACET} we also emphasised the notion of a {\it universal equivariant Poisson deformation} and {\it universal equivariant quantization}. 
Let $X$ be a conical symplectic singularity, and $\Gamma$ a reductive group of $\C^\times$-equivariant Poisson automorphisms.
Then a {$\Gamma$-deformation of $X$} is a flat $\C^\times \times \Gamma$-equivariant morphism $X_S \to S$ of Poisson $\C^\times\times \Gamma$-varieties, where $S$ is equipped with the trivial Poisson structure and trivial $\Gamma$-action, whilst the $\C^\times$-action on $S$ is contracting, and the central fibre of $X_S \to S$ is isomorphic to $X$, via a fixed choice of $\Gamma$-equivariant isomorphism. 
The notion of a $\Gamma$-quantization is defined similarly (see Section~\ref{ss:quantandfunctor}). 
One can consider the functors $\PD^\Gamma_{\C[X]}$ and $\Qnt^\Gamma_{\C[X]}$ of Poisson $\Gamma$-deformations and $\Gamma$-quantizations, and one of the main results of \cite{ACET} states that these functors are representable over the same base, and enjoy the same compatibility as the functors $\PD_{\C[X]}$ and $\Qnt_{\C[X]}$ mentioned above.
Written more informally, every conical symplectic singularity with $\Gamma$-action as above admits a universal Poisson $\Gamma$-deformation (resp.  universal $\Gamma$-quantization), from which every other such deformation (resp. quantization) can be obtained by base change.

In this paper we focus on the the exceptions in the fourth column of Table~\ref{table}.
We build an isomorphism from the nilpotent Slodowy variety to a two-block nilpotent in a symplectic Lie algebra to a certain nilpotent Slodowy variety in an orthogonal Lie algebra, again with two blocks.
We then proceed to examine the Poisson deformation theory of such varieties.
Although the Slodowy slice and the finite $W$-algebra fail to satisfy the  universal properties of Namikawa and Losev they do, in fact, satisfy certain equivariant universal properties with respect to hidden symmetry groups.
This sharpens the observations of Lehn--Namikawa--Sorger \cite[\textsection 12]{LNS}, in a purely algebraic manner.
Since subregular orbits in symplectic Lie algebras have Jordan normal form with two blocks, our results generalize a theorem of Slodowy \cite[Theorem~8.8]{Slo}, and settles a conjecture of \cite{ACET} in type {\sf C}.

\subsection{Poisson isomorphisms and equivariant deformations in type \sf{C}}
We now describe our first main theorem.
 Fix $n \geq 3$ and $0 \leq i < \lfloor \frac{n}{2} \rfloor$.
 Let $e \in \Nc(\so_{2n})$ be an element of Jordan type $(n,n)$ or $(2n- 2i-1, 2i+1)$.
  Also pick an element $e_0 \in \Nc(\sp_{2n-2})$ with Jordan type $(n-1, n-1)$ or $(2n-2i-2, 2i)$ respectively.

\begin{Theorem}
\label{T:zeromaintheorem}
There is an isomorphism of $\C^\times$-Poisson varieties $\Nc_e \simeq \Nc_{e_0}$. 
\end{Theorem}

Using Theorems~1.1 and 1.2 of \cite{ACET}, one can immediately deduce:

\begin{Corollary} \label{C:zeromaincor} \ 
\begin{itemize}\setlength{\itemsep}{2pt}
\item[(i)] $\Ss_e \to \so_{2n} /\!/ \SO_{2n}$ is a universal graded Poisson deformation of $\Nc_{e_0} \subseteq \sp_{2n-2}$.
\item[(ii)] The finite $W$-algebra $U(\so_{2n}, e)$ is a universal filtered quantization of $\Nc_{e_0} \subseteq \sp_{2n-2}$.
\end{itemize}
\end{Corollary}

Theorem~\ref{T:zeromaintheorem} and part (i) of its corollary are algebraic versions of \cite[Proposition~12.1]{LNS}, where the analogous results were proven for germs of complex analytic spaces. Algebraic isomorphisms were obtained \cite[\textsection 5.1]{HL} although Poisson structures were not considered (see also \cite[\textsection 8]{Li} where the isomorphism of Springer resolutions is discussed). Our main result upgrades this to an isomorphism of Poisson varieties.

Our proof of Theorem \ref{T:zeromaintheorem} uses two different arguments according to whether the Jordan blocks of $e_0$ are both even or odd dimensional.
In the even case, we make use of a Poisson presentation for algebras of regular functions on Slodowy slices obtained by the second author in \cite{To}.
This yields to a closed $\C^\times$-Poisson embedding $\Ss_{e_0} \hookrightarrow \Ss_{e}$ (Corollary \ref{C:Poissonsurj}) and a dimension argument allows to conclude (see Section \ref{ss_easycase}). 
In the odd case the Poisson presentation of \cite{To} does not apply, however we can still derive a Poisson embedding of slices using Brown's description \cite{Bw} of the finite $W$-algebra for these Jordan types as a truncated twisted Yangian, see Section \ref{ss_veryeven}.

It is well-known that $\Ss_e \to \so_{2n}/\!/\SO_{2n}$ and $\Ss_{e_0} \to \sp_{2n-2}/\!/\Sp_{2n-2}$ are both $\C^\times$-Poisson deformations of their respective central fibres. 
Likewise, the finite $W$-algebras $U(\so_{2n},e)$ and $U(\sp_{2n-2}, e_0)$ are quantizations of $\C[\Nc_e]$ and $\C[\Nc_{e_0}]$ over their respective centres \cite{PrST}.

To formulate our second main result we must introduce the additional assumption that the Jordan type of $e$ is not $(2k, 2k)$.
Such partitions are called very even, and correspond to two $\SO_{4k}$-orbits permuted by the outer automorphism group of $\so_{4k}$.
See the remarks following \cite[Theorem~5.1.6]{CM} for more detail.

For all other two-block partitions, the outer automorphism group of $\so_{2n}$ stabilises the orbit of $e$, and one can show that there exists a splitting of $\Aut(\so_{2n}) \to \Out(\so_{2n})$ which stabilises the slice $\Ss_e$, see Section~\ref{ss:Gammaonslice}. 
This gives rise to a distinguished group $\Gamma \subseteq \Aut(\Ss_e)$ of $\C^\times$-Poisson automorphisms.
By Theorem~\ref{T:zeromaintheorem}(1) we see that  $\Nc_{e_0}$ admits an $\C^\times\times \Gamma$-action, with $\Gamma$ acting by Poisson automorphisms. 
Note that these symmetries are quite exceptional: they cannot be constructed by restricting automorphisms of $\sp_{2n-2}$.

\begin{Theorem}
\label{T:maintheorem}
Let $n\ge 3$ and let $e_0 \in \Nc(\sp_{2n-2})$ have two Jordan blocks, and suppose that the block sizes are distinct when $n$ is odd.
\begin{enumerate}
\item The Slodowy slice $\Ss_{e_0} \to \sp_{2n-2} /\!/ \Sp_{2n-2}$ is a universal element of the functor $\PD_{\C[\Nc_{e_0}]}^\Gamma$. 
\item The finite $W$-algebra $U(\sp_{2n-2}, e_0)$, viewed as a flat family of filtered algebras over its centre, is a universal element of $\Qnt_{\C[\Nc_{e_0}]}^\Gamma$. 
\end{enumerate}
\end{Theorem}

\begin{Remark}
Note that Theorem~\ref{T:zeromaintheorem} holds for all nilpotent elements of $\sp_{2n-2}$ with at most two Jordan blocks.
On the other hand since the very even orbits in $\so_{2n}$ are not characteristic, the ``hidden symmetries'' of $\Nc_{e_0}$ do not appear when both blocks have odd size, and so Theorem~\ref{T:maintheorem} cannot even be formulated for these orbits.
\end{Remark}

This paper is a natural progression of \cite{ACET}.
Therein, the theory of equivariant deformations and quantizations functors was formalised and applied to the subregular nilpotent Slodowy varieties, which are known to be isomorphic to Kleinian singularities.
Together with our collaborators we showed that when $\g$ is simple of type $\sf B$, $\sf F$ or $\sf C_{2n}$, $n \geq 1$ the Slodowy slice is the universal Poisson graded $\Gamma$-deformation where $\Gamma \subseteq \Aut(\g)$ is a splitting of the outer automorphism group of $\g$. 
We remark that the results of this paper cover all subregular nilpotent Slodowy slices in type $\sf C_n$ for all $n$, extending \cite[Theorem~1.3]{ACET}.

\subsection{Structure of the paper}
In Section \ref{S:CSS}, after fixing notation and conventions, we recall the properties of conical symplectic singularities which are relevant to the current work.
We also introduce functorial formalism for the commutative and non-commutative  equivariant deformation theory of these singularities, surveying the notions of \cite[\textsection 2]{ACET}.
We conclude the section by explaining that the Slodowy slice is a graded Poisson deformation of its nilpotent part, and that an analogous statement holds for the finite $W$-algebra.

In Section \ref{S:tworow} we exhibit a Poisson presentation (Theorem \ref{T:PDyangian}) of the algebras of regular functions on Slodowy slices to certain orbits in type $\sf C$ and $\sf D$.
This immediately leads to a $\C^\times$-Poisson embedding $\Ss_{e_0} \hookrightarrow\Ss_{e}$ mentioned earlier in the introduction, in case $e$ is not very even.

In Section \ref{S:slicesinorthogonal} we discuss various properties of Slodowy slices in orthogonal Lie algebras, again assuming that the Jordan type has two parts, and is not very even. 
We characterise the restriction of the Pfaffian to the slice, amongst all Casimirs.
 We show that the defining ideal of $\Ss_{e_0} \subseteq \Ss_{e}$ is generated by a certain Casimir, written in terms of the Poisson presentation.
 Finally in Section \ref{ss:Gammaonslice}, we describe the Poisson action induced by the outer automorphism group of $\so_{2n}$ on the Slodowy slice.
A key result here is Proposition \ref{P:gammaonC}, illustrating that this action descends to the invariants by a change of sign on the Pfaffian.
This allows us to conclude (see Proposition \ref{P:CisPfaffian}) that the Pfaffian generates the defining ideal of $\Ss_{e_0}$.

Finally, in Section \ref{S:mainres}, we gather together all of the ingredients listed above to give the proofs of Theorems~\ref{T:zeromaintheorem} and \ref{T:maintheorem}.

\subsection{Connections with symplectic duality}
\label{ss:sympduality}
To conclude the introduction we explain that the results of this paper are consistent with some recent conjectures. One of the most exciting themes in the theory of conic symplectic singularities is symplectic duality. This is a conjectural pairing between conic symplectic singularities, which is understood to be an expression of mirror symmetry for 3d $\mathcal{N} = 4$ supersymmetric gauge theories (see \cite{BPLW2} for example). This duality of singularities should transpose certain invariants. For example, there should be an order-reversing bijection on symplectic leaves, and a Koszul duality between certain categories of representations associated to these singularities.

Now let $G$ be a connected reductive algebraic group and $G^\vee$ the Langlands dual group.
As usual, we write $\g = \Lie(G)$ , resp. $\g^\vee = \Lie(G^\vee)$.
There is an order reversing map from nilpotent $G^\vee$-orbits to nilpotent $G$-orbits, known as Barbasch--Vogan--Lusztig--Spaltenstein (BVLS) duality, which restricts to a bijection on special orbits \cite[\textsection 6.3]{CM}. In \cite[\textsection 9]{LMM} this construction was upgraded to a map from nilpotent $G^\vee$-orbits to equivalence classes of $G$-equivariant covers of nilpotent $G$-orbits, which is referred to as {\it refined BVLS duality}. 
It is expected that the nilpotent Slodowy slice in $\g^\vee$ associated to an orbit is symplectic dual to the affinization of the refined BVLS dual orbit cover.

If we take $0 \le i \le \lfloor \frac{n}{2} \rfloor$ and consider the nilpotent slice in $\so_{2n}$ to the orbit with partition $(2n-2i-1, 2i+1)$ then the BVLS dual is the orbit in $\so_{2n}$ with partition $(2^{2i}, 1^{2n-2i})$. This coincides with the refined BVLS dual.

For the nilpotent slice to the orbit of type $(2n-2i-2, 2i)$ in $\sp_{2n-2}$ the BVLS dual is the the orbit of type $(3, 2^{2i-2}, 1^{2n-4i})$ in $\so_{2n-1}$, however the refined BVLS dual of this orbit is the universal (2-fold) cover of this orbit. The affinization of the cover is actually isomorphic (as $\C^\times$-Poisson varieties) to the affinization of the orbit in $\so_{2n}$ with partition $(2^{2i}, 1^{2n-2i})$. This was first observed by Brylinski--Kostant in the language of shared orbit pairs \cite[Theorem~5.9]{BK}, see also \cite[Corollary~2.5(a)]{FJLS} for a recent development on this theme. This classical isomorphism of Poisson varieties is therefore the symplectic dual of our isomorphism in Theorem~\ref{T:zeromaintheorem}.

Finally we remark that there is a more nuanced version of symplectic duality emerging, which takes into account symmetries of singularities, and might be referred to as equivariant symplectic duality, see the introduction to \cite{MMY}. It is expected that our Theorem~\ref{T:maintheorem} is an expression of some aspect of equivariant duality.

\subsection*{Acknowledgements}
We would like to thank Giovanna Carnovale, Ali Craw, Francesco Esposito, Austin Hubbard and Ryo Yamagishi for useful discussions. We would also like to thank Yiqiang Li for drawing our attention to \cite{HL, Li}, and special thanks to Paul Levy and Dmytro Matvieievskyi for useful discussions regarding Section~\ref{ss:sympduality}.
The first author is a member of INDAM-GNSAGA and part of his research work was conducted at Friedrich Schiller Universit\"at Jena.
The second author's work is funded by the UKRI FLF grant numbers MR/S032657/1, MR/S032657/2, MR/S032657/3. 

\section{Deformation theory of conical symplectic singularities}
\label{S:CSS}
We begin the paper by reviewing some of the well-known facts from the theory of conical symplectic singularities, and recalling some of the main results of \cite{ACET}, which built upon the work of Losev and Namikawa \cite{Lo, Na1, Na2}.

\subsection{Conventions}
We work over the field $\C$ of complex numbers.
If $\Gamma$ is a group acting on a set $V$, we denote the action with a dot.
The subset of invariants is $V^\Gamma:= \{ v \in V \mid \gamma \cdot v = v \mbox{ for all } \gamma \in \Gamma \}$. 
If $\Gamma$ is a group acting on an algebra $A$, the algebra of coinvariants is denoted $A_\Gamma$ and is the quotient of $A$ by the ideal generated by $\{a - \gamma \cdot a \mid a \in A, \gamma \in \Gamma\}$, and it should be viewed as the regular functions on the scheme $\Spec(A)^\Gamma$.
Partitions $\lambda = (\lambda_1,...,\lambda_k)$ are ordered in non-decreasing manner: $\lambda_1 \le \lambda_2, \le \cdots \le \lambda_k$.

\subsection{Graded algebras, filtered algebras and representable functors}
Let $\GrAlg$ be the category of finitely generated, non-negatively graded commutative $\C$-algebras $B = \bigoplus_{i \geq 0} B_i$ such that $B_0 = \C$ together with graded morphisms.
For $B \in \GrAlg$, we write $\C_+ := B/B_+$, where $B_+ = \bigoplus_{i \geq 1} B_i$.

By a filtered algebra we mean, unless explicitly stated, an associative non-negatively filtered algebra $A$ with exhaustive connected filtration given by the sequence of subspaces 
$$F_0 A \subset F_1 A \subset \cdots \subset F_i A \subset F_{i+1} A \subset \cdots.$$

We denote by $\FAlg$ the category of finitely generated filtered commutative $\C$-algebras $A = \bigcup_{i \geq 0} F_iA$ such that $F_0 A = \C$ together with strictly filtered morphisms, i.e., $\Hom_{\FAlg}(A,A')$ is the set of filtered algebra morphisms  $A \to A'$ such that $\phi(F_i A) = F_i A' \cap \phi(A)$ for $A, A' \in \FAlg$.
With these assumptions the {\it associated graded} functor
$\gr \colon \FAlg \to \GrAlg$ is exact.
We also have a functor in the opposite direction
$\filt \colon \GrAlg \to \FAlg$ which takes a graded algebra to itself, with filtration induced by the grading, and similar for morphisms. The category $\SFAlg$ of strictly filtered algebras is defined as the essential image of $\filt$.

A Poisson algebra $A$ is a commutative algebra equipped with a Lie bracket $\{ \cdot , \cdot \} \colon A \times A \to A$ which is a biderivation of the associative multiplication, known as a Poisson bracket.
The Poisson centre or Casimirs of $A$ is $\Cas A :=\{a \in A \mid \{a,b\}=0 \mbox{ for all } b \in A\}$.
If a Poisson algebra $A = \bigoplus_{i \geq 0} A_i$ is graded, we say that $A$ has (Poisson) bracket in degree $-d$ if $\{A_i,A_j\} \subset A_{i+j-d}$.
When a Poisson graded algebra $A$ is a $B$-algebra for $B \in \GrAlg$, we implicitly assume that the structure map $B \to A$ has image in $\Cas A$.

We say that a filtered algebra $A$ has commutator in degree $-d$ if $[F_i A,F_j A] \subset F_{i+j-d} A$. In this case the associated graded $\gr A$ naturally inherits the structure of a Poisson algebra with bracket in degree $-d$, see \cite[\S 1.3]{CG}.
When a filtered algebra $A$  is a $B$-algebra for $B \in \FAlg$, we always assume the structure map $B \to A$ is strictly filtered with image in the centre of $A$, which implies that $\gr B \to \Cas (\gr A)$.

\subsection{Conical symplectic singularities with reductive group actions}


A symplectic singularity (introduced by Beauville in \cite{Be}) is a normal variety $X$ whose smooth locus $X^{\reg}$ carries a symplectic form $\pi$ which extends to a regular (possibly degenerate) $2$-form $\tilde \pi$ on some (any) smooth resolution $\tilde X$ of $X$.
The latter condition is independent of the chosen resolution: this can be deduced from \cite[Corollary to Theorem~3.22]{Sh}, since 2-forms on the smooth resolution which are defined away from a locus of codimension 2 can always be extended, using the algebraic form of Hartog's principle.
An affine symplectic singularity with a contracting $\C^\times$-action (such that the Poisson bracket is rescaled by the contracting action) is called a conical symplectic singularity.
Some classical examples of these varieties are the normalizations of closures of nilpotent orbits in simple Lie algebras.

If $X$ is a conical symplectic singularity then it is necessarily affine, and so the sheaf of regular functions is determined by the algebra $A = \C[X]$. 
The choice of the form $\pi$ gives rise to a Poisson bracket on $A$ which is generically non-degenerate.
 Furthermore the $\C^\times$-action endows $A$ with a non-negative grading, and there exists $d \in \N$ such that the Poisson bracket has degree $-d$ with respect to this grading. 
 Therefore the theory of these singularities is encapsulated by an especially nice class of graded Poisson algebras.

In the philosophy of \cite{Slo} we wish to study deformation theory with fixed symmetries. This works especially nicely for conical symplectic singularities, as we observed in \cite{ACET}. In what follows we shall always take $\Gamma$ to be a reductive algebraic group acting on $A = \C[X]$ by graded Poisson automorphisms, 
and we say that $X$ is a \emph{$\Gamma$-conical symplectic singularity}. 

\subsection{Poisson deformations and the deformation functor}
\label{S:Pdef}

Throughout this section, we fix a $\Gamma$-conical symplectic singularity $X$ and  set $A = \C[X]$, with Poisson bracket in degree $-d$.
Let $B \in \GrAlg$.
A graded Poisson deformation of $A$ over $B$ is a pair $(\A, \iota)$ where $\A$ is a graded Poisson $B$-algebra flat as a $B$-module and $\iota \colon \A \otimes_B \C_+ \to A$ is an isomorphism of graded Poisson algebras.
The notion of an isomorphism of graded Poisson deformations of $A$ over the base $B$ is the obvious one (see \cite[Definition 2.4]{ACET}, for example).

Consider the functor $$\PD_A \colon \GrAlg \to \Sets$$ which associates to $B \in \GrAlg$ the set $\PD_A(B)$ of isoclasses of graded Poisson deformations of $A$ over $B$.
For $\beta \colon B \to B'$ in $\GrAlg$ the morphism $\PD_A(\beta)$ maps the isoclass of $(\A, \iota)$ to the isoclass of $(\A', \iota')$, where $\A' = \A \otimes_B B'$ and $\iota'$ is defined by the composition of isomorphisms $\A' \otimes_{B'} \C_+ \to \A \otimes_B \C_+ \to A$ where the second map is $\iota$.

Namikawa \cite{Na1} has shown that there exists $B_u \in \GrAlg$ such that $\PD_A$ is representable over $B_u$ 
i.e. the functors $\PD_A$ and $\Hom_{\GrAlg}(B_u, -)$ are naturally isomorphic.
Let $(\A_u, \iota_u) \in \PD_A(B_u)$ correspond to $\id_{B_u}$: this is called a universal graded Poisson deformation of $A$.
The terminology comes from the fact that $(B_u, (\A_u, \iota_u))$ is a universal element of $\PD_A$ in the notation of \cite[III.1]{MacLane}, however we will always omit the universal base $B_u$ of the deformation for the sake of brevity.
The extent to which a universal element of $\PD_A$ is unique is discussed in \cite[Definition 2.6, Remark 2.7]{ACET}.

For $B \in \GrAlg$, the group $\Gamma$ acts on the set $\PD_A(B)$: indeed, $\gamma \in \Gamma$ maps the isoclass of $(\A, \iota)$ to the isoclass of $(\A, \gamma \circ \iota)$.
The representatives of isoclasses in the fixed-point set $\PD_A(B)^\Gamma$ are called Poisson $\Gamma$-deformations of $A$ over $B$. 
There is a well-defined functor $\PD_A^\Gamma \colon \GrAlg \to \Sets$ mapping $B \in \GrAlg$ to $\PD_A(B)^\Gamma$ and such that $\PD_A^\Gamma(\beta) =\PD_A(\beta)$ for $\beta$ a morphism in $\GrAlg$ \cite[Definiton 2.16]{ACET}.

It follows from the representability of $\PD_A$ that there exists a right $\Gamma$-action on $B_u$ mirroring the $\Gamma$-action on $\PD_A(B_u)$ \cite[Remark 2.18]{ACET}.
Moreover $\PD_A^\Gamma$ is representable over the algebra of coinvariants $(B_u)_\Gamma$ \cite[Proposition 2.23]{ACET} and if $\alpha_\Gamma \colon B_u \to (B_u)_\Gamma$ denotes the quotient morphism in $\GrAlg$, we shall call $\PD_A(\alpha_\Gamma)(\A_u, \iota_u)$ a universal Poisson $\Gamma$-deformation of $A$.

\subsection{Quantizations of Poisson deformations and the quantization functor}
\label{ss:quantandfunctor}
Retain the $\Gamma$-conical symplectic singularity $X$ from the previous section and let $A = \C[X]$, with Poisson bracket in degree $-d$.
The graded Poisson deformation functor has a non-commutative, filtered analogue which we now recall.
Let $B \in \SFAlg$. 
A filtered quantization (of a Poisson deformation) of $A$ over $B$ is a pair $(\Q, \iota)$ where $\Q$ is a filtered $B$-algebra with bracket in degree $-d$, flat as a $B$-module such that $(\gr \Q, \iota) \in \PD_A(\gr B)$ with $\gr B \to \gr \Q$ arising from $B \to \Q$ via the associated graded construction. 
As in the case of graded Poisson deformations, there is a definition of isomorphism between two filtered quantizations over the base $B$  \cite[Definition 2.9]{ACET}.

We have a functor $$\Qnt_A \colon \SFAlg \to \Sets$$ associating to $B \in \SFAlg$ the set $\Qnt_A(B)$ of isoclasses of filtered quantizations of $A$ over $B$. 
For $\beta \colon B \to B'$ in $\SFAlg$, $\Qnt_A(\beta)$ maps the isoclass of $(\Q, \iota)$ to the isoclass of $(\Q', \iota')$, with $\Q' = \Q \otimes_B B'$ and $\iota'$ defined by the composition of isomorphisms $\gr \Q' \otimes_{\gr B'} \C_+ \to \gr \Q \otimes_{\gr B} \C_+ \to A$ where the first map is obvious and the second one is $\iota$.

It follows from the work of Losev \cite{Lo} (see especially Proposition~3.5(1)) that the functors $\PD_A$ and $\Qnt_A$ satisfy an extremely nice compatibility property, and so we say that $A$ admits an  optimal quantization theory.
To be precise, this means that there exists $B_u \in \SFAlg$ such that $\Qnt_A$ is representable over $B_u$, the functor $\PD_A$ is representable over $\gr B_u \in \GrAlg$, and the two natural isomorphisms make the following diagram commute:
\begin{center}
\begin{tikzcd}
\Hom_{\SFAlg}(B_u, -) 
  \arrow[d,,"\texttt{gr}"'] 
   \arrow[r] & 
  \Qnt_{A}(-) \arrow[d,,"\texttt{gr}"] \\
\Hom_{\GrAlg}(\gr B_u, \gr -) \arrow[r,,]                  & \PD_A \gr(-)
\end{tikzcd}
\end{center}
where, for $B \in \SFAlg$, we set $\graded_B \colon \Hom_{\SFAlg}(B_u, B) \to \Hom_{\GrAlg}(\gr B_u, \gr B)$ to be the map $\beta \mapsto \gr \beta$.
We call the element $(\Q_u, \iota_u) \in \Qnt_A(B_u)$ corresponding to $\id_{B_u}$ a universal filtered quantization of $A$, and we remark that it satisfies a property which is analogous to the one of $(\A_u, \iota_u)$.
 See \cite[\textsection 2.5]{ACET} for more detail.

There is also a quantum counterpart to the theory of Poisson $\Gamma$-deformations: for $B \in \SFAlg$, the group $\Gamma$ acts on the set $\Qnt_A(B)$ with $\gamma \in \Gamma$ mapping the isoclass of $(\Q, \iota)$ to the isoclass of $(\Q, \gamma \circ \iota)$.
We call representatives of isoclasses in the fixed-point sets $\Qnt_A(B)^\Gamma$ the filtered $\Gamma$-quantizations of $A$ over $B$. 
This allows to define a functor $\Qnt_A^\Gamma \colon \SFAlg \to \Sets$ whose definition is analogous to $\PD_A^\Gamma$ \cite[Definiton 2.24]{ACET}, and a right $\Gamma$-action on $B_u$ (see \cite[\S 2.7]{ACET}).

The notion of optimal quantization theory can be upgraded to the equivariant setting in the obvious manner: we say that $A$ admits an optimal $\Gamma$-quantization theory if there exists $B_{u,\Gamma} \in \SFAlg$ such that $\Qnt_A^\Gamma$, resp. $\PD_A^\Gamma$, is representable over $B_{u,\Gamma}$, resp. $\gr B_{u,\Gamma} \in \GrAlg$, and the two natural isomorphisms make the following diagram commute:
\begin{center}
\begin{tikzcd}
\Hom_{\SFAlg}(B_{u,\Gamma}, -)   \arrow[d,,"\graded"']  \arrow[r] & \Qnt_{A}^\Gamma(-) \arrow[d,,"\graded"] \\
  \Hom_{\GrAlg}(\gr B_{u,\Gamma}, \gr -) \arrow[r] &  \PD_A^\Gamma \gr(-)
\end{tikzcd}
\end{center}

In \cite[Theorem 2.39]{ACET} we observed that whenever a Poisson algebra $A$ admits an optimal quantization theory it also admits an optimal $\Gamma$-quantization theory, with the coinvariant algebra $(B_u)_\Gamma$ as the representing object $B_{u,\Gamma} \in \SFAlg$.
Finally, if $\alpha_\Gamma \colon B_u \to (B_u)_\Gamma$ is the quotient morphism in $\SFAlg$, then we call $\Qnt_A(\alpha_\Gamma)(\Q_u, \iota_u)$
a universal $\Gamma$-quantization of $A$.

\subsection{The Slodowy slice as a Poisson deformation}
\label{ss:slodowyPoissondeform}

Let $G$ be a complex, connected simple algebraic group and $\g = \Lie(G)$ and let $\Nc:= \Nc(\g)$ be the nilpotent cone of $\g$.
For each $e \in \Nc$ we can choose an $\sl_2$-triple $(e,h,f)$ containig $e$ as the nilpositive element, and define the affine subvariety $\Ss_e := e + \g^f$, known as the Slodowy slice to $G \cdot e$ at $e$. It is a classical result that $\Ss_e $ intersects $G \cdot e$ transversally at $e$.

We now describe a conical structure on $\Ss_e$ and, equivalently, a non-negative grading on $\C[\Ss_e]$. The semisimple derivation $\ad(h)$ has integral eigenvalues and so it defines a grading $\g = \bigoplus_{i \in \Z} \g(i)$, called the Dynkin grading. 
If we let $\nu : \C^\times \to G$ be the unique cocharacter such that $d_1\nu(1) = h$ then we can define a $\C^\times$-action on $\g$ by
\begin{eqnarray}
\label{e:Kazhdanbycochar}
(t, x) \mapsto t^{-2} \Ad(\nu(t)) x
\end{eqnarray}
This action preserves $\Ss_e$ and gives a contracting $\C^\times$-action with unique fixed point $e$, known as the {\it Kazhdan action}. The induced grading on the $\C[\Ss_e]$ is known as the {\it Kazdan grading}.

The Lie bracket on $\g$ extends uniquely to a Lie bracket on $\C[\g^*]$ with $\g$ placed in degree $2$ so that the Poisson bracket is graded in degree $-2$. 
The Casimirs $\Cas \C[\g^*]$ coincide with the $G$-invariant functions $\C[\g^*]^G$. 
Using a choice of  $G$-equivariant isomorphism $\kappa \colon \g \to \g^*$ we transport this to a Poisson structure on $S(\g^*) = \C[\g]$.
Thanks to \cite{GG} the variety $\Ss_e$ can be equipped with a Poisson structure by applying Hamiltonian reduction to $\g$.
\begin{Lemma}
\label{L:casimirlemma}
\cite[Footnote~1]{PrJI}
Restriction $\C[\g] \to \C[\Ss_e]$ gives an isomorphism $\C[\g]^G \simeq \Cas \C[\Ss_e].$
\end{Lemma}
The {\it nilpotent Slodowy variety at $e$} is $\Nc_e := \Nc \cap \Ss_e$. This lies in the union of nilpotent orbits whose closure contain $G \cdot e$ and it is a Poisson subvariety of $\Ss_e$. In fact it carries the structure of a conical symplectic singularity (see \cite[\textsection~3.1]{ACET} for a more detailed survey of these facts).

It is a well known theorem of Kostant (see \cite[\textsection 7]{JaNO} for example) that we have a graded Poisson isomorphism $\C[\g] \otimes_{\C[\g]^G} \C_+ \simeq  \C[\Nc]$, and by \cite[Theorem~5.4]{PrST} we obtain an isomorphism
$\iota \colon \C[\Ss_e] \otimes_{\C[\g]^G} \C_+ \to \C[\Nc_e]$.
To rephrase this in the language of the functor of Poisson deformations, we have that 
$$(\C[\Ss_e],  \i) \in \PD_{\C[\Nc_e]}(\C[\g]^G).$$
 
\subsection{The finite $W$-algebra as a quantization}
\label{ss:finiteW}
Denote by $U(\g)$ the universal enveloping algebra of $\g$, and extend the Kazhdan grading on $\g$ to a filtration on $U(\g)$.

Put $\chi \coloneqq \kappa(e)$.
The set $\g(<-1)_{\chi} := \{ x - \chi(x) \mid x\in \g(<-1)\} \subseteq U(\g)$ is stable under $\ad \g(<0)$. The finite $W$-algebra is the quantum Hamiltonian reduction
$$U(\g,e) := \big(U(\g) / U(\g) \g(<\!-1)_{\chi} \big)^{\ad \g(<0)}.$$

The subquotient inherits an algebra structure from $U(\g)$ with remarkable properties. The Kazhdan filtration descends to $U(\g,e)$ and defines a non-negative filtration on $U(\g,e)$. 
The commutator on $U(\g,e)$ lies in filtered degree $-2$ and $\gr U(\g,e) \simeq \C[\Ss_e]$ as Kazhdan graded Poisson, algebras by \cite[Proposition 5.2]{GG}. Furthermore, $U(\g,e)$ only depends on the adjoint orbit of $e$ up to isomorphism \cite{PrST, GG}.

Denote the centre of $U(\g,e)$ by $Z(\g,e)$. The natural map $U(\g)^G = Z(\g) \to Z(\g,e)$ is an isomorphism (see \cite[Lemma~3.3]{ACET} for example). 
Furthermore the inclusion map $Z(\g,e) \to U(\g,e)$ is flat and strictly filtered with respect to the Kazhdan filtration and $\gr Z(\g) \simeq \C[\g]^G$, where the associated graded is taken with respect to the Kazhdan grading. 
Combining the details above we conclude that $$(U(\g,e),  \i) \in \Qnt_{\C[\Nc_e]}(Z(\g)),$$ where $\iota$ is defined in Section \ref{ss:slodowyPoissondeform}.
We refer the reader to \cite[Lemma~3.4(2)]{ACET} for more detail.

\begin{Theorem}[\cite{LNS} Theorem 1.2 \& 1.3; \cite{ACET} Theorem 1.2]
\label{T:universaltheorem}
Let $e \in \Nc(\g)$ and let $\Nc_e$, resp. $\Ss_e$, resp. $U(\g, e)$ be the nilpotent Slodowy variety, resp. the Sodowy slice, resp. the finite $W$-algebra at $e$. Then the following are equivalent:
\begin{enumerate}
\setlength{\itemsep}{4pt}
\item $\PD_{\C[\Nc_e]}$ is represented by $\C[\g]^G$ and $(\C[\Ss_e], \iota)$ is a universal Poisson deformation;
\item $\Qnt_{\C[\Nc_e]}$ is represented by $Z(\g)$ and $(U(\g,e), \iota)$ is a universal filtered quantization;
\item the pair $(\g, G \cdot e)$ does not appear in Table \ref{table}.
\end{enumerate} 
\end{Theorem}

\section{Presentations of Slodowy slices associated to two row partitions}
\label{S:tworow}

\subsection{Poisson presentations}
Let $X$ be a set. 
The free Lie algebra $L_X$ on $X$ is the initial object in the category of complex Lie algebras generated by $X$ and can be constructed as the Lie subalgebra of the free associative algebra $\C\langle X \rangle$ generated by $X$, with the commutator as a Lie bracket, see \cite[Theorem 4.2]{Serre}.

The {\it free Poisson algebra generated by $X$} is the initial object in the category of complex Poisson algebras generated by $X$. It can be constructed as the symmetric algebra $S(L_X)$. If there is a Poisson surjection $S(L_X) \onto A$ then we say that {\it $A$ is Poisson generated by (the image of) $X$}. 
It is important to distinguish this from $A$ being generated by (the image of) $X$ as a commutative algebra, and we will make this distinction clear whenever necessary.

We say that a complex Poisson algebra $A$ has Poisson generators $X$ and relations $Y \subseteq S(L_X)$  if there is a surjective Poisson homomorphism $S(L_X) \onto A$ and the kernel is the Poisson ideal generated by $Y$.

\subsection{Poisson presentations of Slodowy slices}
\label{ss:Poissonpresentations}

Let $\g$ be a simple Lie algebra  of type ${\mathsf C}_n$, or ${\mathsf D}_n$ and and let $\lambda$ be a partition of $2n$.
Suppose that we are in one of the following situations:
 \begin{enumerate}
 \item[(i)]  $\g$ is of type {\sf C} and all parts of $\lambda$ are even;
 \item[(ii)] $\g$ is of type {\sf D} and all parts of $\lambda$ are odd.
 \end{enumerate}

In \cite{To} the second author provided a Poisson presentation for  $\C[\Ss_e]$ in the cases where $e \in \Nc(\g)$ has partition satisfying (i) or (ii).
In the current setting we only need the presentation in types {\sf C} and {\sf D} when $e$ has two Jordan blocks.
The reader should note that the presentation here is simpler than the general case, in particular relations (1.9)--(1.12) from \cite[Theorem~1.3]{To} are void in the two-block setting, and the relation \eqref{e:symplecticdyrel}, which is special to the symplectic case, can be expressed in a closed form for $n  = 2$ (see \cite[Example~4.9]{To}).

Until otherwise stated let $\g$ be a simple Lie algebra of type ${\mathsf C}_n$ or ${\mathsf D}_n$ and fix an element $e \in \Nc(\g)$ with partion $\lambda = (\lambda_1, \lambda_2)$ satisfying (i) or (ii). 

Write 
$$s_{1,2} = \frac{\lambda_2 - \lambda_1}{2}.$$
Also make the following notation for $r,s \in \Z$
\begin{eqnarray}
\varpi_{r,s} := (-1)^r - (-1)^s = \left\{
\begin{array}{cl} 2 & \text{ if } r \in 2\Z, s\in 2\Z + 1,\\
0 & \text{ if } r + s \in 2\Z,\\
-2 &  \text{ if } r\in 2\Z + 1, s \in 2\Z. \end{array} \right.
\end{eqnarray}

\begin{Theorem} \cite[Theorem~1.3, Proposition~4.7 \& Example~4.9]{To} 
\label{T:PDyangian}
Let $\g$ be a simple Lie algebra of type $\mathsf{C}_n$ or $\mathsf{D}_n$ and let $e \in \Nc(\g)$ with partition $(\lambda_1, \lambda_2) \vdash 2n$ satisfying (i) or (ii).
The algebra $\C[\Ss_e]$ has Poisson generators
\begin{eqnarray}
\label{e:DYgens}
\{\eta_i^{(2r)} \mid 1\le i \le 2, \ r >0\} \cup \{ \theta^{(r)} \mid  r   > s_{1,2}\}
\end{eqnarray}
together with the following relations
\begin{eqnarray}
\label{e:dyrel1}
& & \{\eta_i^{(2r)}, \eta_j^{(2s)}\} = 0 \\ 
\label{e:dyrel2}
& &\big\{\eta_i^{(2r)}, \theta_j^{(s)}\big\} = (\delta_{i,j} - \delta_{i,j+1}) \sum_{t=0}^{r-1} \eta_i^{(2t)} \theta_j^{(2r+s -1- 2t)}
\end{eqnarray}
\begin{eqnarray}
\label{e:dyrel34}
\{\theta^{(r)}, \theta^{(s)}\} = \frac{1}{2} \sum_{t=r}^{s-1} \theta^{(t)} \theta^{(r+s-1-t)} + (-1)^{s_{1,2}}\varpi_{r,s} \sum_{t=0}^{(r+s-1)/2} \eta_{2}^{(r+s-1-2t)} \teta_{1}^{2t} & \text{ for } & r < s.
\end{eqnarray}

\begin{eqnarray}\label{e:dyrel5}
\eta_1^{(2r)} = 0 & \text{ for } & 2r > \lambda_1\\
\label{e:symplecticdyrel}
\sum_{t=0}^{\lambda_1/2} \eta_1^{(2t)} \theta^{(\lambda_1 - 2t + s_{1,2} + 1)} = 0 &  & \text{ when } \g = \sp_{2n}.
\end{eqnarray}
where we adopt the convention $\eta_i^{(0)} = \teta_i^{(0)} = 1$ and the elements $\{\teta_1^{(2r)} \mid r\in \Z_{\ge 0}\}$ are defined via the recursion  
\begin{eqnarray}
\label{e:tetatwisteddefinition}
\teta_1^{(2r)} := -\sum_{t=1}^r \eta_1^{(2t)} \teta_1^{(2r-2t)}.
\end{eqnarray}
Furthermore the Kazhdan grading on $\C[\Ss_e]$ places $\eta_i^{(2r)}$ in degree $4r$ and $\theta^{(r)}$ in degree $2r$, and the Poisson bracket lies in degree $-2$.\hfill\qed
\end{Theorem}

Now let $\g := \so_{2n}$ and $\g_0 := \sp_{2n-2}$ with $n \geq 3$ and fix a partition $\lambda = (\lambda_1, \lambda_2) \vdash 2n$ satisfying (ii) from Section \ref{ss:Poissonpresentations}
 and choose nilpotent elements $e \in \g$ and $e_0 \in \g_0$ with partitions $(\lambda_1, \lambda_2)$ and $(\lambda_1-1, \lambda_2-1)$ respectively.
 In particular, $(\lambda_1-1, \lambda_2-1)$  satisfies (i).
Also choose $\sl_2$-triples for these nilpotent elements and build the Slodowy slices $\Ss_e \subset \g$ and $\Ss_{e_0} \subset \g_0$.

\begin{Corollary}
\label{C:Poissonsurj}
There is a Kazhdan graded surjection of Poisson algebras 
\begin{equation}\label{eq_surj}
\varphi \colon \C[\Ss_e] \to \C[\Ss_{e_0}]
\end{equation}
defined by sending the generators \eqref{e:DYgens} of $\C[\Ss_{e}]$ to those elements of $\C[\Ss_{e_0}]$ with the same labels. 
The kernel of this homomorphism is Poisson generated by the element \eqref{e:symplecticdyrel}.
\end{Corollary}
\begin{proof}
The relations \eqref{e:dyrel1}--\eqref{e:dyrel34} are identical for both $\C[\Ss_{e}]$ and $\C[\Ss_{e_0}]$.
For parity reasons the sets $\{s \mid 2s > \lambda_1\}$ and $\{s \mid 2s > \lambda_1-1\}$ are equal, and so relation \eqref{e:symplecticdyrel} places precisely the same constraints on the generators of $\C[\Ss_{e}]$ and $\C[\Ss_{e_0}]$.

Since the Poisson generator \eqref{e:dyrel5} of the kernel of the map is homogeneous with respect to this grading, it follows that $\varphi$ is graded.
\end{proof}

\section{Slodowy slices in orthogonal algebras}
\label{S:slicesinorthogonal}

\subsection{Invariant theory for classical Lie algebras}
\label{ss:invarianttheory}

Let $n > 0$ and for $x \in \gl_{2n}$, consider the characteristic polynomial $\det(t \mathbb{I}_{2n} - x) = t^{2n} + \sum_{i=1}^{2n}q_i(x)t^{2n-i}\in \C[\gl_{2n}]^{\GL_{2n}}[t]$; for $i = 1, \dots, 2n$, the element $q_i$ is a homogeneous polynomial of degree $i$ in $\C[\gl_{2n}]^{\GL_{2n}}$.
In fact $\{q_i \mid i=1, \dots, 2n\}$ is a complete system of algebraically independent generators of the algebra $\C[\gl_{2n}]^{\GL_{2n}}$ \cite[Proposition 7.9]{JaNO}.
Note that, up to sign, $q_1$ is the trace and $q_{2n}$ is the determinant.

Now let $n \geq 3$ and set $\tilde G:= \OO_{2n}$, $G := \SO_{2n}$, and $\g := \so_{2n}$.
For convenience we assume that these groups stabilise the bilinear form associated to the identity matrix (see Remark~\ref{R:Pfaffrem}).
We recall some aspects of the invariant theory for the adjoint representation of $G$, and we refer the reader to \cite[\S 7.7]{JaNO} for a good introductory survey. 

Consider the restriction of $q_i$ to $\g$, and denote it $p_i$.
 We have $p_i \in \C[\g]^{G}$, and furthermore $p_i = 0$ for $i$ odd, whilst $p_i \ne 0$ for $i$ even.
 It is well-known \cite[\S 7.7]{JaNO} that $p_{2n}$ admits a square root in $\C[\g]^{G}$ known as the {\it Pfaffian}, and we will denote it $p$ throughout this article.
The polynomials $p_2, p_4,...,p_{2n-2}, p$ form a complete set of algebraically independent generators of $\C[\g]^{G}$.

The determinant $\det \colon \tilde G \to \C^\times$ takes values in the subgroup $\{\pm 1\} \subseteq \C^\times$ and $G$ is the kernel of this homomorphism. Therefore the outer automorphism group $\Gamma := \tilde G/G$ of $\g$  is a cyclic group  of order two. 

 Note that the action of $\tilde G$ on $\C[\g]^G$ factors to an action of $\Gamma$, and we can characterise the Pfaffian as follows.
\begin{Lemma}
\label{L:characterisePfaffian}
$\C p$ is the unique one dimensional subspace of $\C[\g]^{G}$ such that:
\begin{enumerate}
\item $\C p$ lies in degree $n$;
\item $\C p$ is isomorphic to the unique non-trivial representation of $\Gamma$.
\end{enumerate}
\end{Lemma}
\begin{proof}
Consider the subalgebra $C \subseteq \C[\g]^{G}$ generated by $p_2, p_4,..., p_{2n}$. We claim that $C = \C[\g]^{\tilde G}$.

Since $\tilde G \subseteq \GL_{2n}$ it follows that every element of $C$ is $\tilde G$-fixed, and so $C \subseteq \C[\g]^{\tilde G}\subseteq \C[\g]^{G}$.
Since $p^2 = p_{2n}$ the quotient $\C[\g]^{G}/C$ is a 2-dimensional $\tilde G$-module with basis $\{ 1, p\}$.
Note that the $\tilde G$-action factors through $\Gamma$ and so $C = \C[\g]^{\tilde G}$ will follow if we demonstrate that $p$ spans a copy of the unique non-trivial simple $\Gamma$-module.

Thanks to our choice of symmetric bilinear form we have $\tilde G = \{ g\in \GL_{2n} \mid g^\top = g^{-1}\}$ and $\g = \{ x\in \gl_{2n} \mid x = - x^\top\}$.
As noted in \cite[\textsection 7.7]{JaNO}, for $g \in \tilde G, x \in \g$ we have $p(gxg^\top) = \det(g) p(x)$, and so $g\cdot p = -p$ for any element $g\in \tilde G$ of determinant $-1$. This shows that $C = \C[\g]^{\tilde G}$.

Since $p^2 = p_{2n}$ it follows that $\C[\g]^G$ is a free $\C[\g]^{\tilde G}$-module of rank two and
$$\C[\g]^G = \C[\g]^{\tilde G} \oplus \C[\g]^{\tilde G} p.$$
By our previous remarks, this is also a decomposition of $\Gamma$-modules: $\C[\g]^{\tilde G}$ is equal to the isotypic component of the trivial $\Gamma$-module, and $\C[\g]^{\tilde G}p$ is the isotypic component of the non-trivial simple $\Gamma$-module.
Since $\C p$ is the unique one dimensional subspace of $\C[\g]^{\tilde G} p$ of degree $n$ the proof is complete.
\end{proof}

\begin{Remark}
\label{R:Pfaffrem}
In this section we could have chosen $\tilde G$ to be the subgroup stabilising any non-degenerate symmetric bilinear form on $\C^{2n}$. 
Such forms are classified by invertible symmetric matrices, and if we chose one other than the identity matrix then our discussion of the Pfaffian would be slightly more complicated, see \cite[\textsection 7.7]{JaNO}.
By the remarks of \cite[\textsection 1.3]{JaNO} we see that Lemma~\ref{L:characterisePfaffian} holds true for any choice of form.
\end{Remark}

Let $e \in \Nc(\g)$ and embed it in a $\sl_2$ triple $(e,h,f)$.
Retain the notation $\Ss_e = e + \g^f$ for the Slodowy slice.

\begin{Lemma}
\label{L:identifypfaff}
Suppose that the short exact sequence $G \into \tilde G \onto \Gamma$ is split by a map $s \colon \Gamma \to \tilde G$ such that $s(\Gamma)$ preserves $\Ss_e$.
Then the restriction of the Pfaffian to $\Ss_e$ generates the unique one dimensional subspace $\C p|_{\Ss_e} \subseteq \Cas \C[\Ss_e]$ such that:
\begin{enumerate}
\item $\C p|_{\Ss_e}$ has Kazhdan degree $2n$;
\item $\C p|_{\Ss_e}$ is isomorphic to the unique non-trivial representation of $s(\Gamma)$.
\end{enumerate}
\end{Lemma}

\begin{proof}
Using \eqref{e:Kazhdanbycochar} we see that elements of $\C[\g]^G$ of total degree $i$ lie in Kazhdan degree $2i$. Now the result is a direct consequence of Lemma~\ref{L:characterisePfaffian}, along with the remarks of Section~\ref{ss:slodowyPoissondeform}.
\end{proof}

\subsection{A special Casimir on the orthogonal slice}
\label{s:Pffafianrestrcited}
In this section we fix $\g := \so_{2n}$ with $n\geq 3$ and an element $e \in \Nc(\g)$ with associated partition $\lambda = (\lambda_1, \lambda_2) \vdash 2n$ satisfying (ii) from Section \ref{ss:Poissonpresentations}.
We will study the slice $\Ss_e$ using the Poisson presentation of Theorem~\ref{T:PDyangian}, and we resume all the relevant notation from that section. In particular we let $s_{1,2} := \frac{\lambda_2 - \lambda_1}{2}$, so that $n = \lambda_1 + s_{1,2}$.

Introduce the notation
\begin{eqnarray}
\label{e:Cdefn}
z := \sum_{t = 0}^{\frac{\lambda_1 - 1}{2}} \eta_1^{(2t)} \theta^{(\lambda_1 - 2t + s_{1,2})} = \sum_{t = 0}^{\frac{\lambda_1 - 1}{2}} \eta_1^{(2t)} \theta^{(n - 2t)} \in \C[\Ss_e].
\end{eqnarray}
 
The goal of this section is to prove:
\begin{Proposition}
\label{P:Cincentre}
We have $z \in \Cas \C[\Ss_e]$.
\end{Proposition}

\begin{proof}
We proceed by direct computation, showing that Poisson brackets between $z$ and the generators \eqref{e:DYgens} vanish. For all integers $r \geq 0$, we have:

\begin{eqnarray*}
\big\{ \eta_1^{(2r)}, z\big\} &=& \sum_{t = 0}^{\frac{\lambda_1 - 1}{2}} \big\{ \eta_1^{(2r)},  \eta_1^{(2t)} \theta^{(n- 2t)} \big\} = \sum_{t = 0}^{\frac{\lambda_1 - 1}{2}} \eta_1^{(2t)} \big\{ \eta_1^{(2r)},   \theta^{(n - 2t)} \big\} =\\
&=& \sum_{t = 0}^{\frac{\lambda_1 - 1}{2}} \eta_1^{(2t)} \sum_{k = 0}^{r - 1}  \eta_1^{(2k)}   \theta^{(2r + n - 2t -2k)} =
 \sum_{k = 0}^{r - 1} \eta_1^{(2k)}   \sum_{t = 0}^{\frac{\lambda_1 - 1}{2}}    \eta_1^{(2t)} \theta^{(2r + n - 2t  -2k)} =\\
&=& \sum_{k = 0}^{r - 1} \eta_1^{(2k)}   \big\{  \eta_1^{(\lambda_1 + 1)}, \theta^{(n+ 2r -2k)} \big\} = 0,
\end{eqnarray*}
where we use $\eta_1^{(\lambda_1 + 1)} = 0$ by \eqref{e:dyrel5} in the final line of the calculation. We now show that $\big\{ \eta_2^{(2r)}, z\big\}  = 0$ for all $r\ge 0$, in a similar fashion:
\begin{eqnarray*}
\big\{ \eta_2^{(2r)}, z\big\} &=& \sum_{t = 0}^{\frac{\lambda_1 - 1}{2}} \big\{ \eta_2^{(2r)},  \eta_1^{(2t)} \theta^{(n - 2t)} \big\} =
 -\sum_{t = 0}^{\frac{\lambda_1 - 1}{2}} \eta_1^{(2t)} \big\{ \eta_2^{(2r)},   \theta^{(n - 2t)} \big\} =\\
 &=& -\sum_{k = 0}^{r - 1} \eta_2^{(2k)}   \big\{  \eta_1^{(\lambda_1 + 1)}, \theta^{(n +2r -2k)} \big\} = 0.
\end{eqnarray*}

Now we observe that
$\{\eta_1^{(2)}, \theta^{(s)}\} = \theta^{(s+1)}$ for all $s \geq s_{1,2} + 1$. 
Using the Jacobi identity and a recursive argument, we deduce that $\{z, \theta^{(s)}\} = 0$ for all $s \geq s_{1,2} + 1$ provided $\{ z, \theta^{(s_{1,2} + 1)} \} = 0$.
Using $s_{1,2} = n - \lambda_1$, it remains to show that  $\big\{ z, \theta^{(n - \lambda_1 + 1)} \big\} = 0$.
\begin{eqnarray*}
\left\{  z, \theta^{(n - \lambda_1 + 1)}\right\} &=& \sum_{t = 0}^{\frac{\lambda_1 - 1}{2}} \left\{   \eta_1^{(2t)} \theta^{(n- 2t)} , \theta^{(n - \lambda_1 + 1)}\right\}  \\
&=& \sum_{t = 0}^{\frac{\lambda_1 - 1}{2}} \left\{  \eta_1^{(2t)},  \theta^{(n - \lambda_1 + 1)}  \right\} \theta^{(n -2t)} + \sum_{t = 0}^{\frac{\lambda_1 - 1}{2}}  \eta_1^{(2t)}  \left\{   \theta^{(n-2t)}, \theta^{(n - \lambda_1 + 1)} \right\} \\
&= & \sum_{t = 1}^{\frac{\lambda_1 - 1}{2}} \left\{  \eta_1^{(2t)},  \theta^{(n - \lambda_1 + 1)}  \right\} \theta^{(n -2t)} - \sum_{t = 0}^{\frac{\lambda_1 - 1}{2}}  \eta_1^{(2t)}  \left\{ \theta^{(n - \lambda_1 + 1)},   \theta^{(n-2t)} \right\}\\
&=& \sum_{t = 1}^{\frac{\lambda_1 - 1}{2}} \sum_{k = 0}^{t-1}  \eta_1^{(2k)}  \theta^{(2t +n - \lambda_1 -2k)}  \theta^{(n - 2t )}
- \sum_{t = 0}^{\frac{\lambda_1 - 1}{2}}  \eta_1^{(2t)}  \left\{ \theta^{(n - \lambda_1 + 1)},   \theta^{(n-2t)} \right\} \\
&=& \sum_{t = 1}^{\frac{\lambda_1 - 1}{2}} \sum_{k = 0}^{t-1}  \eta_1^{(2k)}  \theta^{(2t +n - \lambda_1 -2k)}  \theta^{(n - 2t )}
- \sum_{t = 0}^{\frac{\lambda_1 - 3}{2}}  \eta_1^{(2t)}  \left\{ \theta^{(n - \lambda_1 + 1)},   \theta^{(n-2t)} \right\} \\
&=& \sum_{k = 0}^{\frac{\lambda_1 - 3}{2}}  \eta_1^{(2k)}   \sum_{t = k+1}^{\frac{\lambda_1 - 1}{2}}   \theta^{(2t +n - \lambda_1 -2k)}  \theta^{(n - 2t )}
- \sum_{t = 0}^{\frac{\lambda_1 - 3}{2}}  \eta_1^{(2t)}  \left\{ \theta^{(n - \lambda_1 + 1)},   \theta^{(n-2t)} \right\}.
\end{eqnarray*}

It is therefore enough to prove that, for fixed $0 \leq k \leq \frac{\lambda_1 - 3}{2}$, 
\begin{equation}\label{eq_toprove}
\sum_{j = k+1}^{\frac{\lambda_1 - 1}{2}}   \theta^{(2j +n - \lambda_1 -2k)}  \theta^{(n - 2j)} = \left\{ \theta^{(n - \lambda_1 + 1)},   \theta^{(n-2k)} \right\}
\end{equation}
Since $n - \lambda_1 + 1$ and $ n - 2k$ have the same parity, relation \eqref{e:dyrel34} implies that the right hand side of \eqref{eq_toprove} equals:
$$ 
\dfrac{1}{2}  \sum_{i = n - \lambda_1 + 1}^{n- 2k -1} \theta^{(i)} \theta^{( \lambda_2 -2k  -i)}=   \theta^{(n-\lambda_1 + 1)} \theta^{( n -2k  -1)} + \cdots + \theta^{(\frac{\lambda_2 - 1 -2k}{2})}\theta^{(\frac{\lambda_2 +1 -2k}{2})}
= \sum_{t = 1}^{\frac{\lambda_1 - 1 -2k}{2}} \theta^{(n - \lambda_1 + t)} \theta^{(  n -2k -t)}.$$


For $k = 0, \dots, \frac{\lambda_1 - 3}{2}$ we define the  index sets for the two summations
$$I_L = \left\{j \in \Z \mid k+1 \leq j \leq \dfrac{\lambda_1 - 1}{2} \right\}, \quad
 I_R = \left\{t \in \Z \mid 1 \leq t \leq \frac{\lambda_1 - 1 - 2k}{2} \right\}.$$


\textbf{Case 1.} If  $\frac{\lambda_1 - 1 -2k}{2}$ is odd, there exists $\bar j \in I_L$ such that $m - 2 \bar j = \frac{\lambda_2 - 1 -2k}{2}$, namely $\bar j =     \frac{\lambda_1 + 1 +2k}{4}$.
%

Then:
 $$\sum_{j \in I_L} \theta^{(2j +n - \lambda_1 -2k)}  \theta^{(n - 2j)} = \sum_{j = k+1}^{\bar j -1}   \theta^{(2j +n - \lambda_1 -2k)}  \theta^{(n - 2j)} + \sum_{j = \bar j}^{\frac{\lambda_1 - 1}{2}}   \theta^{(2j +n - \lambda_1 -2k)}  \theta^{(n - 2j)}.$$

By comparing pieces, one sees that:
$$\sum_{j = k+1}^{\bar j -1}   \theta^{(2j +n - \lambda_1 -2k)}  \theta^{(n - 2j)} = 
\sum_{2 \Z \cap I_R} \theta^{(n - \lambda_1 + t)} \theta^{(  n -2k -t)}$$
and since $\C[\Ss_e]$ is commutative
$$\sum_{j = \bar j}^{\frac{\lambda_1 - 1}{2}}   \theta^{(2j +n - \lambda_1 -2k)}  \theta^{(n - 2j)} = 
\sum_{(2 \Z +1) \cap I_R} \theta^{(n - \lambda_1 + t)} \theta^{(n -2k -t)},$$
this concludes the proof in this case.


\textbf{Case 2.} If $\frac{\lambda_1 - 1 -2k}{2}$ is even, there exists
 $\bar j \in I_L$ such that 
$2 \bar j + n - \lambda_1 -2k = \frac{\lambda_2 - 1 -2k}{2}$, namely 
$\bar j =  \frac{\lambda_1 -1 +2k}{4}$.
Then the proof mimics the one in the previous case: we observe that
$$\sum_{j = k+1}^{\bar j }   \theta^{(2j +n - \lambda_1 -2k)}  \theta_{(n - 2j)} = 
\sum_{2 \Z \cap I_R} \theta^{(n - \lambda_1 + t)} \theta^{( n -2k -t)}$$
and applying commutativity in $\C[\Ss_e]$ we have
$$\sum_{j = \bar j + 1}^{\frac{\lambda_1 - 1}{2}}   \theta^{(2j +n - \lambda_1 -2k)}  \theta^{(n - 2j)} = 
\sum_{(2 \Z +1) \cap I_R} \theta^{(n - \lambda_1 + t)} \theta^{( n -2k -t)}.$$
\end{proof}


Now since $z$ is a Casimir, the Poisson ideal it generates coincides with the ideal which it generates when viewing $\C[\Ss_e]$ as a commutative algebra.
By Corollary~\ref{C:Poissonsurj}, and using the same notation we initiated there, we have:
\begin{Corollary}
\label{C:kernelcor}
The kernel of the homomorphism $\varphi \colon \C[\Ss_e] \to \C[\Ss_{e_0}]$ defined in \eqref{eq_surj} is generated by the single element $z$.$\hfill\qed$
\end{Corollary}

\begin{Remark}
\label{R:regularremark}
Note that the results of this section  are valid also when $\lambda_1 = 1$: Theorem~\ref{T:PDyangian} holds for $n \in \{1, 2\}$ however the relations are significantly simpler in the case $n = 1$.
\end{Remark}

\subsection{The $\Gamma$-action on the orthogonal slice}
\label{ss:Gammaonslice}
Throughout this section, we fix $\g := \so_{2n}, n\geq 3$ and we let $e \in \Nc(\g)$ with a two-block partition satisfying (ii) from Section \ref{ss:Poissonpresentations} and $z$ as in \eqref{e:Cdefn}.
Our goal is to examine the action of the outer automorphism group of $\g$ on $z$, with respect to a carefully chosen splitting of the surjection $\OO_{2n} \onto \OO_{2n} /\SO_{2n}$.
The set up which we use to make the necessary calculations is  explained in \cite[Part I, \textsection 4]{To}. 
For ease of reference we recap some of the main facts here.

We begin by introducing a choice of embedding $\g \subseteq \gl_{2n}$ and a nilpotent element $e \in \g$, a block diagonal nilpotent matrix with Jordan blocks of sizes $\lambda_1, \lambda_2$, with $1$ or $0$ on the super-diagonal and zeroes elsewhere.
This is described in detail in \cite[\textsection 4.2]{To}.

The natural representation $\C^{2n}$ of $\g$ has basis $\{ b_{i,j} \mid 1\le i \le 2, \ 1\le j \le \lambda_i\}$.
 The general linear Lie algebra $\gl_{2n}$ has a basis 
\begin{equation}
\label{e:glbasis}
\{e_{i,j;k,l} \mid 1\le i,k \le 2, \ 1\le j \le \lambda_i, \ 1 \le l \le \lambda_k\}
\end{equation}
where $e_{i,j; k,l} b_{r,s} = \delta_{k,r} \delta_{l, s} b_{i,j}$,
The relations in $\gl_{2n}$ are given by the following formula
\begin{eqnarray}
& [e_{i_1,j_1;k_1,l_1}, e_{i_2,j_2;k_2,l_2}] = \delta_{k_1, i_2} \delta_{l_1, j_2} e_{i_1, j_1; k_2, l_2} - \delta_{k_2, i_1}\delta_{l_2, j_1} e_{i_2, j_2; k_1, l_1}
\end{eqnarray}
and $\gl_{2n}$ admits an involution $\tau : \gl_{2n} \to \gl_{2n}$ determined by 
\begin{eqnarray}
\label{e:taudefn}
\tau(e_{i,j;k,l}) = (-1)^{j - l - 1}  e_{k, \lambda_k + 1 - l; i, \lambda_i + 1 - j}.
\end{eqnarray}
The subalgebra $\g := \gl_{2n}^\tau$ is isomorphic to $\so_{2n}$, see \cite[(4.13)]{To}.

The nilpotent element
\begin{eqnarray}
\label{e:edefinition}
e := \sum_{i=1}^2 \sum_{j = 1}^{\lambda_i-1} e_{i,j;i,j + 1}
\end{eqnarray}
has two Jordan blocks of sizes $\lambda_1$ and $\lambda_2$. Furthermore $e$ forms part of an $\sl_2$-triple $\{e,h,f\} \subseteq \g$ and so $\tau$ fixes each element of this triple.

By the classification of $\sl_2$-representations, the operator $\ad(h)$ defines  Dynkin gradings $\g = \bigoplus_{i\in \Z} \g(i)$ and $\gl_{2n} = \bigoplus_{i\in \Z} \gl_{2n}(i)$. To be precise we have 
\begin{eqnarray}
\label{e:gooddegrees}
\deg(e_{i,j;k,l}) = 2(l-j) + \lambda_i - \lambda_k.
\end{eqnarray}
by \cite[(4.5)]{To}. Since $\lambda_1 - \lambda_2$ is even, it follows that the Dynkin grading is even, i.e. $\g(i) = \gl_{2n}(i) = 0$ for $i \notin 2\Z$.

We now introduce a splitting of $\OO_{2n} \onto \OO_{2n}/\SO_{2n}$.
Let $\gamma \in \GL_{2n}$ be the element defined by
\begin{eqnarray*}
\gamma(b_{i,j}) := \left\{ \begin{array}{cl} -b_{i,j} & \text{ if } i =1; \\ b_{i,j} & \text{ if } i=2.\end{array}\right. 
\end{eqnarray*}

\begin{Lemma}
Retain notation from the current section. Then:
\label{L:gammalemma}
\begin{enumerate}
\setlength{\itemsep}{4pt}
\item $\det(\gamma) = -1$ and $\gamma^2$ is the identity.
\item For all $i,j,k,l$ as per \eqref{e:glbasis} we have $$\Ad(\gamma) (e_{i,j;k,l}) = \left\{ \begin{array}{cl} -e_{i,j;k,l} & \text{ if } i \ne k; \\ e_{i,j;k,l} & \text{ if } i=k.\end{array}\right. $$
\item $\gamma$ lies in the orthogonal group $\OO_{2n}\subseteq \GL_{2n}$ satisfying $\Lie(\OO_{2n}) = \g$.
\item If $\widetilde \Gamma \subseteq \OO_{2n}$ denotes the subgroup generated by $\gamma$ then the map $\widetilde \Gamma \to \OO_{2n}$ splits the surjection $\OO_{2n} \onto \Gamma := \OO_{2n} / \SO_{2n}$. Furthermore $\widetilde \Gamma$ fixes $e,h$ and $f$.
\end{enumerate}
\end{Lemma}
\begin{proof}
Part (1) and (2) follow directly from the definitions. 

Comparing (2) with formula \eqref{e:taudefn} we see that $\Ad(\gamma)$ commutes with $\tau$, hence $\Ad(\gamma)$ preserves $\g$. Now (3) follows from the fact that $\OO_{2n} \subseteq \GL_{2n}$ is precisely the subgroup of $\GL_{2n}$ preserving $\g$ under $\Ad$.

As we noted earlier $\Gamma$ is a cyclic group of order two generated by the coset of any element of $\OO_{2n}$ of determinant $-1$, and so the first part of (4) follows from (1) and (3).

Examining \eqref{e:edefinition} and \cite[(4.5)]{To} we see that $\Ad(\gamma)$ fixes both $e$ and $h$, and by \cite[Lemma~3.4.4]{CM} we see that $\Ad(\gamma)$ also fixes $f$.
\end{proof}

\begin{Remark} \label{rk_orc}
If $K$ is a  simple  algebraic group with $\mathfrak{k} = \Lie (K)$
 and $e \in \Nc(\mathfrak{k})$ with an attached $\sl_2$-triple $(e,h,f)$, Slodowy defines the \emph{outer reductive centralizer} $\CA(e)$ of $e$ to be the elements of $\Aut(\mathfrak{k})$ fixing the triple pointwise (cf. \cite[\S 7.6]{Slo}). 
This is a possibly disconnected reductive group which acts via graded Poisson automorphisms on the Slodowy slice $e + \mathfrak{k}^f$, and its nilpotent part as well, see the argument for \cite[Lemma 4.1]{ACET}.
If $e$ belongs to a characteristic nilpotent orbit, then the first part of the proof of \cite[\S 7.6, Lemma 2]{Slo} carries over to this more general setting and yields the short exact sequence
$1 \to C_K(e,h,f) \to \CA(e) \to \Out(\mathfrak{k}) \to 1$, where $C_K(e,h,f)$ is the pointwise stabiliser in $K$ of the triple and $\Out(\mathfrak{k})$ denotes the outer automorphism group of $\mathfrak{k}$.

For $\mathfrak{k} = \so_{2n}$ and $e$ as in the current section, one can show that $\widetilde{\Gamma}$ is isomorphic to a subgroup of $\CA(e)$ splitting the map $\CA(e) \to \Out(\mathfrak{k})$.
\end{Remark}

There is a parabolic Lie algebra $\p =\bigoplus_{i \ge 0} \g(i)$ constructed from the Dynkin grading.
Denote the nilradical of the opposite parabolic algebra by $\n = \bigoplus_{i<0} \g(i)$.
We recall some notation and facts from Sections~\ref{ss:slodowyPoissondeform} and \ref{ss:finiteW}: let $\chi = \kappa(e) \in \g^*$ and write $\n_\chi := \{x - \chi(x) \mid x\in \n\}\subseteq S(\g) \cong \C[\g^*]$.
 We introduced the Kazhdan grading on $\C[\g^*]$, defined by placing $\g(i)$ in degree $i + 2$. 

The decomposition $S(\g) = S(\p) \oplus S(\g) \n_\chi$ (see \cite[Formula (8.4)]{BKshift}) gives rise to a projection operator $S(\g) \onto S(\p)$ which is $\Ad(\gamma)$-equivariant and Kazhdan graded, and this leads to a twisted $\n$-action on $S(\p)$ given by the composition
$$S(\p) \overset{\ad(x)}\longrightarrow S(\g) \to S(\p)$$
for $x\in \n$. 
We denote the invariants with respect to this action by $S(\p)^{\tw(\n)}$. 
It follows from \cite[Theorem~4.1]{GG}, and is recapped with more detail in \cite[\textsection 1]{BKshift}, that we have the following isomorphism
\begin{eqnarray}
\label{e:SandC}
S(\p)^{\tw(\n)} \cong \C[\Ss_e].
\end{eqnarray}
This depends entirely on the fact that the Dynkin grading is even. It is important to note that \eqref{e:SandC} respects the Kazhdan gradings, and is equivariant with respect to the group of automorphisms of $\g$ which fix the $\sl_2$-triple $(e,h,f)$.

These same remarks can all be carried out for the general linear Lie algebra: there is a parabolic subalgebra $\op = \gl_{2n}(\ge\! 0)$ and the opposite nilradical $\on = \gl_{2n}(<\! 0)$ admits a twisted action on $S(\op)$, such that
$$\C[e + \gl_{2n}^f] \cong S(\op)^{\tw(\on)}$$
as Kazhdan graded Poisson algebras. The involution $\tau$ defined above extends to an involution on $S(\op)$ which preserves the $\tw(\on)$-invariants. For concision we write $S := S(\op)^{\tw \on}$.

\begin{Lemma} \cite[(2.2) \& Theorem~2.6]{To}
\label{L:optop}
Let $\op_- \subseteq \op$ denote the subspace spanned by $x\in \op$ such that $\tau(x) = -x$. 
Consider the $\Ad(\gamma)$-equivariant surjective algebra homomorphism $S(\op) \to S(\p) = S(\op) / (\op_-)$. This restricts to a surjective Poisson algebra homomorphism $S^\tau \onto S(\p)^{\tw(\n)}.$
\end{Lemma}

Now we describe elements of $S^\tau$ which map to the generators $\eta_i^{(2r)}$ and $\theta^{(r)}$ of $S(\p)^{\tw(\n)}$ described in Theorem~\ref{T:PDyangian}, under the homomorphism appearing in Lemma~\ref{L:optop}. The following elements were first introduced in the enveloping algebra $U(\op)$ in a slightly different language in \cite[\textsection 9, (1)--(6)]{BKshift}. The current set-up is identical to \cite[\textsection 4.3]{To}.

For $1 \le i,k \le 2$, $0 \le x \le 1$ and $r > 0$, we let
\begin{eqnarray*}
\label{e:tdefn}
& & t_{i,k;x}^{(r)} := \sum_{s=1}^r (-1)^{r-s} \sum_{\substack{(i_m, j_m, k_m, l_m) \\ m =1,...,s}} (-1)^{\hash \{q=1,\dots,s-1 \mid k_q \le x\}}
e_{i_1, j_1; k_1, l_1} \cdots e_{i_s, j_s; k_s, l_s} \in S(\op),
\end{eqnarray*}
where the sum is taken over all indexes such that for $m=1,...,s$
\begin{eqnarray}
\label{e:indexbounds}
1 \le i_1,...,i_s, k_1,...,k_s \le n, 1\le j_m \le \lambda_{i_m} \text{ and } 1\le l_m \le \lambda_{k_m}
\end{eqnarray}
satisfying the following six conditions:
\begin{itemize}
\setlength{\itemsep}{4pt}
\item[(a)] $\sum_{m=1}^s (2l_m - 2j_m + \lambda_{i_m} - \lambda_{k_m}) = 2(r - s)$;
\item[(b)] $2l_m- 2j_m + \lambda_{i_m} - \lambda_{k_m} \ge 0$ for each $m = 1,\dots,s$;
\item[(c)] if $k_m > x$, then $l_m < j_{m+1}$ for each $m = 1,\dots,s-1$;
\item[(d)] if $k_m \le x$ then $l_m \ge j_{m+1}$ for each $m = 1,\dots,s-1$; 
\item[(e)] $i_1 = i$, $k_s = k$;
\item[(f)] $k_m = i_{m+1}$ for each $m = 1,\dots,s-1$.
\end{itemize}

\begin{Lemma}
\label{L:poissonsurjdefn}
These satisfy the following properties:
\begin{enumerate}
\item the elements $t_{i,k;x}^{(r)}$ lie in $S(\op)^{\tw(\on)}$.
\item the elements $t_{i,i;i-1}^{(2r)}$ and $t_{1, 2;1}^{(r)} + (-1)^{r + \frac{\lambda_2 - \lambda_1}{2}}t_{2,1;1}^{(r)}$ are invariant under the action of $\tau$ on $S(\op)$.
\item The map $S^\tau \to S(\p)^{\tw(\n)} = \C[\Ss_e]$ from Lemma~\ref{L:optop} has the following effect:
\begin{eqnarray*}
t_{i,i;i-1}^{(2r)} & \longmapsto & \eta_i^{(2r)},\\
t_{1, 2;1}^{(r)} + (-1)^{r + \frac{\lambda_2 - \lambda_1}{2}}t_{2,1;1}^{(r)} & \longmapsto & \theta^{(r)}.
\end{eqnarray*}
In other words the elements on the left hand side map to elements of $\C[\Ss_e]$ satisfying the relations of the symbols on the right, which are given in Theorem~\ref{T:PDyangian}.
\end{enumerate}
\end{Lemma}
\begin{proof}
Part (1) and (2) both follow directly from \cite[Proposition~4.5]{To}. Part (3) follows from formulas (3.59), (3.61), (4.24) and Propositions~4.7 and 4.8 of {\it op. cit.}
\end{proof}

\begin{Proposition}
\label{P:gammaonC}
We have $\Ad(\gamma) \theta^{(r)} = -\theta^{(r)}$ and $\Ad(\gamma)\eta_i^{(2r)} = \eta_i^{(2r)}$ for $i=1,2$.
 If we define $z\in \C[\Ss_e]$ via formula \eqref{e:Cdefn} then $$\Ad(\gamma)z = -z.$$
\end{Proposition}
\begin{proof}
By \eqref{e:gooddegrees} and Lemma~\ref{L:gammalemma}(2) we see that $\Ad(\gamma)$ preserves $S(\op)$.

Now using Lemma~\ref{L:gammalemma}(2) and the combinatorial condition (e) appearing in the definition of $t_{i,k;x}^{(r)}$, it is easy to see that factors $e_{i,j;k,l}$ with $i \ne k$ occur with even parity in every summand of the expressions for $t_{i,i;i-1}^{(2r)}$, whilst factors $e_{i,j;k,l}$ with $i \ne k$ occur with odd parity in the expressions for both $t_{1, 2;1}^{(r)}$ and $t_{2,1;1}^{(r)}$. 
By Lemma~\ref{L:gammalemma}(2) we see that $\Ad(\gamma)(t_{i,i;i-1}^{(2r)}) = t_{i,i;i-1}^{(2r)}$ and $\tau(t_{1, 2;1}^{(r)} + (-1)^{r + \frac{\lambda_2 - \lambda_1}{2}}t_{2,1;1}^{(r)}) = - (t_{1, 2;1}^{(r)} + (-1)^{r + \frac{\lambda_2 - \lambda_1}{2}}t_{2,1;1}^{(r)})$.

Since the map $S^\tau \to S(\p)^{\tw(\n)}$ described in Lemma~\ref{L:optop} is $\Ad(\gamma)$-equivariant, it follows from Lemma~\ref{L:poissonsurjdefn}(3) that $\Ad(\gamma)$ acts as claimed on $\theta^{(r)}, \eta_i^{(2r)} \in S(\p)^{\tw(\n)} = \C[\Ss_e]$.
Now the claim $\Ad(\gamma)z = -z$ is a direct consequence of \eqref{e:Cdefn}.
\end{proof}

Recall the Pfaffian $p \in \C[\g]^G$ from Section~\ref{ss:invarianttheory}, and write $p|_{\Ss_e}$ for the restriction to the slice.

\begin{Proposition}
\label{P:CisPfaffian}
We have $z = p|_{\Ss_e}$ up to a non-zero scalar.
\end{Proposition}
\begin{proof}
By Proposition~\ref{P:Cincentre} we know that $z$ is Poisson central in $\C[\Ss_e]$.
 By Theorem~\ref{T:PDyangian} we also know that $z$ lies in Kazhdan degree $2n$.
  The proof concludes by combining Lemma~\ref{L:identifypfaff}, Lemma~\ref{L:gammalemma}(4) and Proposition~\ref{P:gammaonC}.
\end{proof}

\section{Exceptional isomorphisms and equivariant quantizations in type {\sf C}} \label{S:mainres}
\subsection{The case of partitions with two even parts} \label{ss_easycase}
We now fix the notation required to prove the main theorems of the paper.
Let $n \geq 3$, and set $\tilde G \coloneqq \OO_{2n}$, $ G \coloneqq \SO_{2n}$ and $G_0 = \Sp_{2n-2}$.
Put $\g \coloneqq \Lie G$ and $\g_0 \coloneqq \Lie G_0$.

Fix a partition $(\lambda_1, \lambda_2) \vdash 2n$ satisfying (ii) of Section~\ref{ss:Poissonpresentations}. Pick elements $e \in \Nc(\g)$ and $e_0 \in \Nc(\g_0)$ with orbits labelled by by $(\lambda_1 , \lambda_2)$ and $(\lambda_1-1, \lambda_2-1)$ respectively.  Write $\Nc_e := \Ss_e \cap \Nc(\g)$ and $\Nc_{e_0} := \Ss_{e_0} \cap \Nc(\g_0)$.

\begin{Theorem}
\label{T:isomorphiccones}
The Poisson varieties $\Nc_e \subset \g$ and $\Nc_{e_0} \subset \g_0$ are isomorphic. 
\end{Theorem}
\begin{proof}
We  have the following equalities for dimensions of orbits.
\begin{eqnarray*}
 \dim G \cdot e = \dim \g - \dim \g^e, &&
\dim G_0 \cdot e_0 = \dim \g_0- \dim \g_0^{e_0}.
\end{eqnarray*}
Using the transversality of the Slodowy slices we have that $\dim \Nc_e = \codim_{\Nc(\g)} (G \cdot e)$, and similarly for $\Nc_{e_0}$. 
It now follows from \cite[\textsection 3.3(3)]{JaNO} that
\begin{eqnarray}
\label{e:Ndims}
\dim \Nc_e = \lambda_1 = \dim \Nc_{e_0}.
\end{eqnarray}

Write $\varphi \colon \C[\Ss_e] \to \C[\Ss_{e_0}]$  for the Kazhdan graded homomorphism \eqref{eq_surj}.
 If we denote by $I \subseteq \Cas \C[\Ss_e]$ the unique maximal graded ideal of the Poisson centre, and similarly for $I_0 \subseteq \Cas \C[\Ss_{e_0}]$, then we have $\varphi(I) \subseteq I_0$. 

It follows from the remarks of Section~\ref{ss:slodowyPoissondeform} that $I$ generates the defining ideal of $\Nc_e$ inside $\C[\Ss_e]$ and similarly $I_0$ generates the defining ideal of $\Nc_{e_0}$ inside $\C[\Ss_{e_0}]$.
Now from the remarks of the previous paragraph we see that the morphism  of Poisson varieties $\Ss_{e_0} \to \Ss_e$ induced by $\varphi$ restricts to a closed embedding $\Nc_{e_0} \hookrightarrow \Nc_e$.
By \cite[Theorem~5.4(ii)]{PrST} we know that $\Nc_e$ is an irreducible variety, and so from \eqref{e:Ndims} it follows that $\Nc_{e_0} \to \Nc_e$ is an isomorphism of Poisson varieties.
\end{proof}

\subsection{Twisted Yangians and the case of two odd parts}\label{ss_veryeven}
Keep the notation $\g = \so_{2n}$ and $\g_0 = \sp_{2n-2}$, and now assume $n \geq 4$ is even.
Let $\lambda = (n, n)$, let $\lambda_0 = (n-1, n-1)$ and pick nilpotent elements $e \in \Nc(\g)$ and $e_0 \in \Nc(\g_0)$ in orbits with partitions $\lambda$ and $\lambda_0$ respectively. We remark that there are precisely two $\SO_{2n}$-orbits with partition $\lambda$ (see the remarks following \cite[Theorem~5.1.6]{CM}) and $e$ may be chosen in either of these.
We prove the last remaining case of Theorem~\ref{T:zeromaintheorem} (1).

\begin{Theorem}
\label{T:lastbit}
$\Nc_{e_0} \cong \Nc_e$ as Poisson $\C^\times$-varieties.
\end{Theorem}

The method we used in Theorem~\ref{T:isomorphiccones} depended entirely on Theorem~\ref{T:PDyangian}, which only applies in the case where $\lambda$ has parts of odd size. In order to carry out the argument in our setting, we recall a quantum analogue of Theorem~\ref{T:PDyangian} which applies to nilpotent elements which have all Jordan blocks of the same size. This leads to an analogue of Corollary~\ref{C:Poissonsurj} in our setting.

The twisted Yangian $Y_{n}^-$ is a non-commutative algebra introduced by Olshanski \cite{O}, see \cite{Mo} for a good survey of its properties. It is generated by symbols $S_{i,j}^{(r)}$ with $i,j = 1,2,..,n$ and $r > 0$ subject to the {\it quaternary} and {\it symmetry} relations \cite[(2.6), (2.7)]{Mo}. The canonical filtration on $Y_n^-$ places $S_{i,j}^{(r)}$ in degree $r$.

The following is a special case of \cite[Theorem~1.2 \& 2.3]{Bw}
\begin{Theorem}
\label{T:Brownstheorem}
\begin{enumerate}
\item There is a surjective filtered algebra homomorphism 
$Y_2^- \onto U(\g,e)$
with kernel generated by
\begin{eqnarray}
\label{e:ker1}
\{S_{i,j}^{(r)} \mid r > n\}.
\end{eqnarray}

\item There is a surjective filtered algebra homomorphism $Y_2^-  \onto  U(\g_0,e_0)$ with kernel generated by
\begin{eqnarray}
\label{e:ker2}
\{S_{i,j}^{(r)}  - \frac{1}{2} S_{i,j}^{(r-1)} \mid r > n-1\}.
\end{eqnarray}
\end{enumerate}
\end{Theorem}

We let $y_2^- = \gr Y_2^-$ be the associated graded algebra with respect to the canonical filtration. It is a commutative algebra and comes equipped with a Poisson structure, see \cite[(2.2)]{ACET}, \cite[Remark~2.4.5]{Mo}. From the above we can now deduce the existence of required map between Slodowy slices.

Let $(e,h,f)$ and $(e_0, h_0, f_0)$ be choices of $\sl_2$-triples for $e$ and $e_0$ respectively.

\begin{Corollary}
\label{C:veryevenembedding}
There is a $\C^\times$-equivariant closed Poisson embedding $e_0 + \g_0^{f_0} \hookrightarrow e + \g^f$.
\end{Corollary}
\begin{proof}
Consider the associated graded maps of the surjections appearing in Theorem~\ref{T:Brownstheorem}. Write $I$ for the kernel of $y_2^- \onto \C[e+\g^f]$ and $I_0$ for the kernel of the map $y_2^- \onto \C[e_0 + \g_0^{f_0}]$. We claim that we have an inclusion of graded ideals $I \subseteq I_0$, and this will imply the corollary.

Write $s_{i,j}^{(r)}$ for the image of $S_{i,j}^{(r)} $ in $y_2^-$. The claim will follow if we can show that $I_0$ is generated  by $\{s_{i,j}^{(r)} \mid r > n\}$ and $I_0$ is generated by $\{s_{i,j}^{(r)} \mid r > n-1\}$, as associative algebra ideals. Examining \eqref{e:ker1} and \eqref{e:ker2} it is clear that the specified elements lie in $I$ and $I_0$ respectively. The fact that they generate can be easily deduced from \cite[Theorem~2.3]{Bw}.
\end{proof}

\begin{proofof}
Following the same line of reasoning as the first part of the proof of Theorem~\ref{T:isomorphiccones} we see that $\dim \Nc_{e_0} = \dim \Nc_{e}$. It follows that the inclusion from Corollary~\ref{C:veryevenembedding} is actually an isomorphism. \hfill\qed
\end{proofof}

\subsection{Equivariant deformations and quantizations}
In this final section we prove Theorem~\ref{T:maintheorem}.
Let $n \geq 3$, and put $\g \coloneqq \so_{2n}$ and $\g_0 \coloneqq \sp_{2n-2}$.
Fix a partition $\lambda = (\lambda_1, \lambda_2) \vdash 2n$ satisfying (ii) from Section~\ref{ss:Poissonpresentations}, so that $\lambda_0 = (\lambda_1-1, \lambda_2-1)$ satisfies (i). Pick $\sl_2$-triples $(e,h,f)$ in $\g$ and $(e_0, h_0, f_0)$ in $\g_0$ for these partitions.

In Sections~\ref{ss:slodowyPoissondeform} and \ref{ss:finiteW}, we explained that $(\C[\Ss_{e_0}],  \i) \in \PD_{\C[\Nc_{e_0}]}(\C[\g_0]^{G_0})$ and $(U(\g_0,e_0),  \i) \in \Qnt_{\C[\Nc_{e_0}]}(\C[Z(\g_0)])$.

By Lemma~\ref{L:gammalemma}(4) there is a subgroup $\Gamma\subset \tilde G$ which splits the outer automorphism group of $\g$, and fixes $(e,h,f)$ pointwise.
 The $\Gamma$-action on $\Ss_e$ stabilises $\C[\g]^G$.

\begin{equation}\label{eq:restr_cas}
\psi \colon \C[\g]^{G} \to \Cas \C[\Ss_e] \overset{\varphi}{\longrightarrow} \Cas \C[\Ss_{e_0}] \to \C[\g_0]^{G_0}
\end{equation}
where the  homomorphism $\varphi$ is described in \eqref{eq_surj}, and the other two maps arise from Lemma~\ref{L:casimirlemma}.

\begin{Lemma}
\label{L:coinvandvarphi}
The map $\psi$ is surjective and its kernel coincides with the kernel of the quotient of $\C[\g]^G$ to the $\Gamma$-coinvariants, therefore the algebras 
$\C[\g_0]^{G_0}$ and  $(\C[\g]^G)_\Gamma$ are graded isomorphic.
\end{Lemma}

\begin{proof}
We have the equality $\ker \varphi \cap \Cas \C[\Ss_e] =  (z)$, by Corollary~\ref{C:kernelcor}, whilst by Proposition~\ref{P:CisPfaffian} we have $(p|_{\Ss_e}) = (z)$ as ideals in $\C[\Ss_e]$.
This implies that $\ker \psi = (p)$, where $p$ denotes the Pfaffian.
Since $p$ is a homogeneous generator of $\C[\g]^G$, and since the dimensions of the graded components of  $\C[\g]^G$ and $\C[\g_0]^{G_0}$ are uniquely determined by the degrees of the homogeneous generators, we obtain surjectivity of the map $\psi$ by comparison of the graded components.
By Lemma~\ref{L:characterisePfaffian} we see that the algebra of $\Gamma$-coinvariants $(\C[\g]^G)_\Gamma$ is obtained from $\C[\g]^G$ after quotienting with the ideal $(p)$. This completes the argument.
\end{proof}

By Theorem~\ref{T:isomorphiccones} we may regard $\Gamma$ as a group of Kazhdan graded Poisson automorphisms of $\C[\Nc_e] \simeq \C[\Nc_{e_0}]$.
Thus we may consider the fixed-points functors $\PD^\Gamma_{\C[\Nc_{e_0}]}$ and $\Qnt^\Gamma_{\C[\Nc_{e_0}]}$ from Sections \ref{S:Pdef} and \ref{ss:quantandfunctor}.

\begin{Remark}
Observe that, in the above situation, $e$ is regular if and only if $e_0$ is regular, if and only if $\Nc_e = \{e\}$  if and only if $\Nc_{e_0} = \{e_0\}$, so Theorem \ref{T:isomorphiccones} holds trivially.
Moreover, in such cases the deformation theory of $\Nc_{e_0} \simeq \Nc_e$  degenerates to a triviality, see also the remark following \cite[Theorem 3.5]{ACET}.
For this reason, in what follows, we exclude the regular case.
\end{Remark}

We now prove Theorem~\ref{T:maintheorem}.

\begin{Theorem}
\label{T:PDGammaoddcases}
Let $\g_0, e_0$ be as fixed in this section, and suppose that $e_0$ is not regular.
\begin{enumerate}
\item $\PD^\Gamma_{\C[\Nc_{e_0}]}$ is represented by $\C[\g_0]^{G_0}$  and $(\C[\Ss_{e_0}], \iota_0)$ is a universal element.
\item $\Qnt^\Gamma_{\C[\Nc_{e_0}]}$ is represented by $Z(\g_0)$ and $(U(\g_0,e_0), \iota_0)$ is a universal element.
\end{enumerate}
\end{Theorem}
\begin{proof}

Since $\Nc_{e}$ is a conical symplectic singularity the functors $\PD_{\C[\Nc_{e}]}$ and $\Qnt_{\C[\Nc_{e}]}$ are both representable (Section~\ref{ss:quantandfunctor}).
In \cite[Theorem~3.5 \& 3.6]{ACET} it was demonstrated that they are both represented over $\C[\g]^G$ and $Z(\g)$ respectively, and that we have the following universal elements (see Sections~\ref{ss:slodowyPoissondeform} and \ref{ss:finiteW}):
\begin{itemize}
\item[(i)] $(\C[\Ss_e], \i) \in \PD_{\C[\Nc_{e}]}(\C[\g]^G)$
\item[(ii)] $(U(\g,e), \iota) \in \Qnt_{\C[\Nc_e]}(Z(\g))$
\end{itemize}

By Theorem~\ref{T:isomorphiccones} the same remarks hold if we replace $e$ by $e_0$. 

As we explained in Section~\ref{ss:quantandfunctor} the functor $\PD^\Gamma_{\C[\Nc_{e_0}]}$ is represented over the base $(\C[\g]^G)_\Gamma$. 
Furthermore an universal element is given by the base change of $(\C[\Ss_e], \i)$ through the natural map $\C[\g]^G \to (\C[\g]^G)_\Gamma$. 
Now part (1) of the Theorem follows from Lemma~\ref{L:coinvandvarphi}.

Since $U(\g_0, e)$ is a filtered quantization of $\C[\Ss_{e_0}]$ part (2) of the Theorem follows from \cite[Theorem~1.1]{ACET}.
\end{proof}

\begin{Remark} \label{rk_orc_ve}
Although $\Nc_e \cong \Nc_{e_0}$ when $e$ has Jordan type $(n,n)$, $e_0$ has type $(n-1, n-1)$ and $n$ is even (by Theorem~\ref{T:lastbit}), there can be no analogue of Theorem~\ref{T:PDGammaoddcases} formulated for a group of automorphisms $\Gamma \subseteq \CA(e)$ (see Remark~\ref{rk_orc} for notation).
The orbit of $e$ is not characteristic in $\g$ and the outer reductive centralizer satisfies $\CA(e) \subset \SO_{2n}$.
Hence the action of $\CA(e)$ on $\C[\Ss_e]$ restricts to a trivial action on the Poisson centre, and for any subgroup $\Gamma \subset \CA(e)$ the universal element for $\PD^\Gamma(\C[\Nc_{e_0}])$ is $\C[\Ss_e]$. The same remarks apply in the non-commutative setting, thanks to \cite[Theorem~1.1]{ACET}.

It would be interesting to know whether $\Nc_{e_0}$ is equipped with any symmetries which do not come from restricting elements of $\Aut(\sp_{2n-2})$ or of $\Aut(\so_{2n})$, which could serve as a replacement for $\Gamma$ when trying to formulate such a statement.
\end{Remark}

\end{document}